\documentclass[10pt, reqno]{amsart}
\usepackage{graphicx, amssymb, amsmath, amsthm}
\numberwithin{equation}{section}

\let\Re=\undefined\DeclareMathOperator*{\Re}{Re}

\newcommand{\R}{\mathbb{R}}

\newtheorem{theorem}{Theorem}[section]

\newtheorem{lemma}[theorem]{Lemma}

\newtheorem{corollary}[theorem]{Corollary}

\newtheorem{proposition}[theorem]{Proposition}

\theoremstyle{definition}
\newtheorem{definition}[theorem]{Definition}
\newtheorem{remark}[theorem]{Remark}

\theoremstyle{remark}

\makeatletter
\newcommand{\Extend}[5]{\ext@arrow0099{\arrowfill@#1#2#3}{#4}{#5}}
\makeatother

\begin{document}
\title[Klein-Gordon-Hartree Equation]
{The Defocusing Energy-critical  Klein-Gordon-Hartree Equation}

\author{Qianyun Miao}
\address{School of Mathematics and System Sciences, Beihang University,
Beijing,\ China,\ 100191 } \email{mqy8955@sina.com}

\author{Jiqiang Zheng}
\address{Laboratoire de Math\'ematiques J.A. Dieudonn\'e,
Universit\'e Nice Sophia-Antipolis, 06108 Nice Cedex 02, France}
\email{zhengjiqiang@gmail.com, zheng@unice.fr}

\begin{abstract}
In this paper, we study the scattering theory  for the defocusing
energy-critical Klein-Gordon equation with a cubic convolution
$u_{tt}-\Delta u+u+(|x|^{-4}\ast|u|^2)u=0$ in the spatial dimension
$d \geq 5$. We utilize the strategy in \cite{IMN} derived from
concentration compactness ideas to show that the proof of the global
well-posedness and scattering is reduced to disprove the existence
of the soliton-like solution. Employing technique from \cite{Pa2},
we consider a virial-type identity in the direction orthogonal to
the momentum vector so as to exclude such solution.
\end{abstract}

 \maketitle

\begin{center}
 \begin{minipage}{100mm}
 \end{minipage}
 \end{center}


\section{Introduction}
\setcounter{section}{1}\setcounter{equation}{0} This paper is
devoted to the study of the Cauchy problem of the defocusing
energy-critical Klein-Gordon-Hartree equation
\begin{align} \label{equ1}
\begin{cases}    \ddot{u} - \Delta u  +  u +  f(u)=  0,  \qquad  (t,x) \in
\mathbb{R}\times\mathbb{R}^d, d \geq 5,\\
u(0,x)=u_0(x),~u_t(0,x)=u_1,\end{cases}
\end{align}
where $f(u)=(V(x)*|u|^2) u $ with $V(x)=|x|^{-4}$. Here $u$ is a
real-valued function defined in $\mathbb{R}^{d+1}$, the dot denotes
the time derivative, $\Delta$ is the Laplacian in $\mathbb{R}^{d}$,
$V(x)$ is called the potential, and $*$ denotes the spatial
convolution in $\mathbb{R}^{d}$.

Formally, the solution $u$ of \eqref{equ1} conserves  the energy
\begin{equation*}\label{econ}
\aligned E(u(t),\dot{u}(t))=&\frac12 \int_{\mathbb{R}^d}
\big(\big|\dot{u}(t,x) \big|^2 +\big| \nabla u(t,x) \big|^2 + \big|
u(t,x) \big|^2 \big)dx\\& + \frac{1}{4}
\iint_{\mathbb{R}^d\times\mathbb{R}^d}
 \frac{|
u(t,x)|^2 |u(t,y)|^2}{|x-y|^4} dxdy\\
=&E(u_0,u_1),
\endaligned
\end{equation*}
and the momentum
\begin{equation}
P(u)(t)=\int_{\mathbb{R}^d}u_t(t,x)\nabla u(t,x) dx=P(u)(0).
\end{equation}

For the equation $(\ref{equ1})$ with nonlinearity
$f(u)=\mu(|x|^{-\gamma}\ast|u|^2)u,~\mu=\pm1$, using the ideas of
Strauss \cite{St81a}, \cite{St81b} and Pecher \cite{Pe85}, Mochizuki
\cite{Mo89} showed that if $d \geq 3$, $2\leq \gamma < \min(d, 4)$,
then global well-posedness and scattering results with small data
hold in the energy space $H^1(\R^d)\times L^2(\R^d)$. For the
general initial data, we refer to the authors \cite{MZh} where we
develop  a complete scattering theory in the energy space for
\eqref{equ1} with the subcritical nonlinearity (i.e. $2< \gamma <
\min(d, 4)$) for both defocusing ($\mu=1$) and focusing ($\mu=-1$)
in spatial dimension $d\geq3$. In this paper, we will focus on the
energy-critical case, i.e. $\gamma=4$ and $d\geq5.$ We refer also
to Miao-Zhang \cite{MZ} where the low regularity for the cubic
convolution defocusing Klein-Gordon-Hartree equation is discussed.

Before stating our main results, we recall the scattering theory for
the classical Klein-Gordon equation, i.e \eqref{equ1} with
nonlinearity $f(u)=\mu|u|^{p-1}u$. For $\mu=1$ and
\begin{equation}\label{}
   1+\frac{4}{d}<p<1+\frac{4\gamma_d}{d-2},
\quad
 \gamma_d=\left\{ \aligned
    &1 ,& 3\leq d\leq 9; \\
     &\frac{d}{d+1},& d\geq 10,
\endaligned
\right.
\end{equation}
Brenner \cite{Br85} established the scattering results in the energy
space in dimension $d\geq 10$. Thereafter, Ginibre and Velo
\cite{GiV85b} exploited the Birman-Solomjak space $\ell^m(L^q,I,B)$
in \cite{BiS75} and the delicate estimates to improve the results in
\cite{Br85}, which covered all subcritical cases. Finally K.
Nakanishi \cite{Na99b} obtained the scattering results for the
energy-critical case by the strategy of induction on energy
\cite{CKSTT07} and a new Morawetz-type estimate. And recently, S.
Ibrahim, N. Masmoudi and K. Nakanishi\cite{IMN, IMN1} utilized the
concentration compactness ideas to give the scattering threshold for
the focusing (i.e. $\mu=-1$) nonlinear Klein-Gordon equation. We
remark that their method also works for the defocusing case. We will
utilize their argument to study the scattering theory for the
defocusing energy-critical Klein-Gordon-Hartree equation.

On the other hand, the scattering theory for the Hartree equation
$$i\dot{u}=-\Delta u+(|x|^{-\gamma}*|u|^2)u$$
has been also  studied by many authors (see
\cite{GiV00,LiMZ09,MXZ07a,MXZ09,MXZ09a,MXZ07b}). For the
energy-subcritical
 case, i.e. $\gamma<4$, Ginibre and Velo \cite{GiV00} obtained the asymptotic completeness
in the energy space $H^1(\R^d)$ by deriving the associated Morawetz
inequality and extracting an useful Birman-Solomjak type estimate.
Nakanishi \cite{Na99d} improved the results by a new Morawetz
estimate. For the energy-critical case ($\gamma=4$ and $d\geq5$),
Miao, Xu and Zhao \cite{MXZ07a} took advantage of a new kind of the
localized Morawetz estimate to rule out the possibility of the
energy concentration at origin and established the scattering
results in the energy space for the radial data. We refer also to
\cite{MXZ09,MXZ09a,MXZ07b} for the general data and also
mass-critical case.

 Compared with the classical
Klein-Gordon equation with the local nonlinearity $f(u)=|u|^{p-1}u$,
the nonlinearity $f(u)=(V(\cdot)*|u|^2) u $ is nonlocal, which
brings us many difficulties. The main difficulty is the absence of a
Lorentz invariance which could be used to control the momentum
efficiently. We will overcome this difficulty by considering a
Virial-type identity in the direction orthogonal to the momentum
vector following the technique in \cite{Pa2}.

Now we introduce  the definition of the strong solution for
\eqref{equ1}.
\begin{definition}[solution]\label{def1.1}
 A function $u:~I\times\R^d\to\R$ on a nonempty time
interval $0\in I$ is a strong solution to \eqref{equ1} if for any
compact $J\subset I$, $(u,u_t)\in C_t^0(J; H^1(\R^d)\times
L^2(\R^d))$ and $$u\in W(J),\quad W(I):=
L_t^{\frac{2(d+1)}{d-1}}(J;B^{\frac12}_{\frac{2(d+1)}{d-1},2}(\mathbb{R}^d))$$
 and for each $t\in I$, $(u(t),\dot u(t))$ satisfies  the  following
Duhamel's formula:
\begin{equation}\label{duhamel}
{u(t)\choose \dot{u}(t)} = V_0(t){u_0(x) \choose u_1(x)}
-\int^{t}_{0}V_0(t-s){0 \choose f(u(s))} ds,
\end{equation}
where
$$  V_0(t) = {\dot{K}(t), K(t)
\choose \ddot{K}(t), \dot{K}(t)}, \quad
K(t)=\frac{\sin(t\omega)}{\omega},\quad
\omega=\big(1-\Delta\big)^{1/2}.$$  The interval $I$ is called the
lifespan of $u$. Moreover, if the solution $u$ cannot be extended to
any strictly large interval, then we say that $u$ is a
maximal-lifespan solution. We say that $u$ is a global solution if
$I=\R.$
\end{definition}

\begin{remark}
From Remark \ref{rem2.3} below, we obtain the solution $u$ lies in
the space $W(I)$ locally in time. Also, the finiteness of the norm
on maximal-lifespan implies the solution is global and scatters in
both time directions by standard argument. In view of this, we
define
\begin{equation}\label{scattersize}
S_I(u)=\|u\|_{ST(I)}=\|u\|_{[W](I)}
\end{equation}
as the scattering size of $u$.

\end{remark}

Our main result is the following global well-posedness and
scattering result in the energy space.

\begin{theorem}\label{theorem}
Assume that  $d\geq 5$,  and $(u_0,u_1)\in H^1(\mathbb{R}^d)\times
L^2(\mathbb{R}^d)$. Then there exists a unique global solution
$u(t)$ of \eqref{equ1} which scatters in the sense that
 there exist  solutions $v_\pm$ of the free Klein-Gordon equation
\begin{equation}\label{le}
    \ddot{v} - \Delta v  + v =  0
\end{equation}
with $(v_\pm(0), \dot{v}_\pm(0))\in H^1\times L^2$ such that
\begin{equation}\label{1.2}
\big\|\big(u(t), \dot{u}(t)\big)-\big(v_\pm(t),
\dot{v}_\pm(t)\big)\big\|_{H^1\times L^2} \longrightarrow 0,\quad
\text{as}\quad t\longrightarrow \pm\infty.\end{equation}
\end{theorem}

The outline for the proof of Theorem \ref{theorem}:
 we define the function $\Lambda$ by
\begin{equation}\label{znorm}
\Lambda(E)=\sup\{\Vert u\Vert_{ST(I)}:E(u,u_t)\le E\}
\end{equation}
where the supremum is taken over all strong solutions $u$ of
\eqref{equ1} on any interval $I$ with energy not greater than $E$,
and define
\begin{align*}\label{emax}E_{max}&=\sup\{E:\Lambda(E)<+\infty\}.
\end{align*}
The small data scattering (Theorem \ref{small} below) tells us
$E_{max}>0$. Our goal next is to prove that $E_{max}=+\infty$. We
argue by contradiction. We show that if $E_{max}<+\infty$, then
there exists a nonlinear solution of \eqref{equ1} with energy be
exactly $E_{max}$. Moreover, this solution satisfies some strong
compactness properties. This is completed in Section 4 where we
utilize the profile decomposition that was established in
\cite{IMN}, and a strategy introduced by Kenig and Merle \cite{KM}.
We consider a virial-type identity in the direction orthogonal to
the momentum vector  following the technique \cite{Pa2} to obtain a
contradiction. We refer to Section 5 for more details.

The paper is organized as follows. In Section $2$, we deal with the
local theory
 for the equation $(\ref{equ1})$.  In Section $3$, we give
the linear and nonlinear profile decomposition and show some
properties of the profile. In Section 4, we extract a critical
solution. Finally in Section $5$, we preclude the critical solution,
which completes the proof of Theorem \ref{theorem}.

\section{Preliminaries}
 \setcounter{section}{2}\setcounter{equation}{0}

\subsection{Notation}
First, we give some notations which will be used throughout this
paper. We always assume the spatial dimension $d\geq 5$ and let
$2^*=\frac{2d}{d-2}$. For any $r: 1\leq r \leq \infty$, we denote by
$\|\cdot \|_{r}$ the norm in $L^{r}=L^{r}(\mathbb{R}^d)$ and by $r'$
the conjugate exponent defined by $\frac{1}{r} + \frac{1}{r'}=1$.
For any $s\in \mathbb{R}$, we denote by $H^s(\mathbb{R}^d)$ the
usual Sobolev space. Let $\psi\in \mathcal{S}(\mathbb{R}^d)$ be such
that $\text{supp}\ {\widehat{\psi}} \subseteq \big\{\xi: \frac{1}{2}
\leq|\xi| \leq 2 \big\}$ and $ \sum_{j\in \mathbb{Z}} \widehat{\psi}
(2^{-j} \xi) = 1 $ for $\xi \neq 0.$ Define $\psi_0$ by
$\widehat{\psi}_0 = 1 -
 \sum_{j\geq 1} \widehat{\psi} (2^{-j} \xi).$  Thus $\text{supp}\
\widehat{\psi}_0 \subseteq \big\{\xi: |\xi| \leq 2 \big\}$ and
$\widehat{\psi}_0 = 1$ for $|\xi| \leq 1$. We denote by $\Delta_j$
and $\mathcal{P}_0$ the convolution operators whose symbols are
respectively given by $\widehat{\psi}(\xi/2^{j})$ and
$\widehat{\psi}_0(\xi)$. For $s \in \mathbb{R}, 1\leq r \leq
\infty$, the inhomogeneous Besov space $ B^{s}_{r, 2}(\mathbb{R}^d)$
is defined by
$$ B^{s}_{r, 2}(\mathbb{R}^d) = \bigg\{ u \in
\mathcal{S}'(\mathbb{R}^d), \|\mathcal{P}_0 u\|^2_{L^r}+
\big\|2^{js} \|\Delta_j u\|_{L^r} \big\|^2_{l^2_{j\in \mathbb{N}}} <
\infty \bigg\}.$$ For details of Besov space, we refer to
\cite{BL76}. For any interval $I\subset \mathbb{R}$ and any Banach
space $X$ we denote by ${\mathcal C}(I; X)$ the space of strongly
continuous functions from $I$ to $X$ and by $L^q(I; X)$ the space of
strongly measurable functions from $I$ to $X$ with $\|u(\cdot);
X\|\in L^q(I).$ Given $d,$ we define, for $2\le r\le \infty$,
$$\delta(r)=d\Big(\frac12-\frac1{r}\Big).$$
Sometimes abbreviate $\delta(r)$, $\delta(r_i) $ to
$\delta,~\delta_i $ respectively. We denote by $\langle\cdot,
\cdot\rangle$ the scalar product in $L^2$. We let $L_*^p$ denote the
weak $L^p$ space.

 \subsection{Strichartz estimate}
In this section, we consider the Cauchy problem  for the equation
$(\ref{equ1})$
\begin{equation} \label{equ2}
    \left\{ \aligned &\ddot{u} - \Delta u  +  u +  f(u)=  0, \\
    &u(0)=u_0,~\dot{u}(0)=u_1.
    \endaligned
    \right.
\end{equation}
The integral equation for the Cauchy problem $(\ref{equ2})$ can be
written as
\begin{equation}\label{inte1}
u(t)=\dot{K}(t)u_0 + K(t)u_1-\int^{t}_{0}K(t-s)f(u(s))ds,
\end{equation}
or
\begin{equation}\label{inte2}
{u(t)\choose \dot{u}(t)} = V_0(t){u_0(x) \choose u_1(x)}
-\int^{t}_{0}V_0(t-s){0 \choose f(u(s))} ds,
\end{equation}
where
$$K(t)=\frac{\sin(t\omega)}{\omega}, \quad V_0(t) = {\dot{K}(t), K(t)
\choose \ddot{K}(t), \dot{K}(t)}, \quad \omega=\big(
1-\Delta\big)^{1/2}.$$

Let $U(t)=e^{it\omega}$, then
\begin{equation*}
\dot{K}(t)= \frac{U(t)+U(-t)}{2}, \qquad  K(t)=
\frac{U(t)-U(-t)}{2i\omega}.
\end{equation*}

Now we recall the following dispersive estimate for the operator
$U(t)=e^{it\omega}$.
\begin{lemma}[\cite{Br85,GiV85b}]\label{lem21}
Let $2\leq r\leq \infty$ and $0\leq \theta\leq 1$. Then
\begin{equation*}
\big\|e^{i\omega t}f
\big\|_{B^{-(d+1+\theta)(\frac12-\frac1r)/2}_{r, 2}} \leq \mu(t)
\big\|f\big\|_{B^{(d+1+\theta)(\frac12-\frac1r)/2}_{r', 2}},
\end{equation*}
where
\begin{equation*}
\mu(t)=C \min\bigg\{ |t|^{-(d-1-\theta)(\frac12-\frac1r)_{+} },
|t|^{-(d-1+\theta)(\frac12-\frac1r)}\bigg\}.
\end{equation*}
\end{lemma}

Combining the above lemma,  the abstract duality and interpolation
argument(see \cite{GiV95,KeT98}), we have the following Strichartz
estimates.
\begin{lemma}[\cite{Br85,GiV85b,MZF}]\label{lem22}
Let $0\leq\theta_i \leq 1$, $\rho_i \in \mathbb{R}$, $2\leq q_i, r_i
\leq \infty,~i=1,2$. Assume that $(\theta_i,d,q_i,r_i)\neq
(0,3,2,\infty)$ satisfy the following admissible conditions
\begin{equation}
\left\{ \aligned \label{rl} 0\leq \frac{2}{q_i} &\leq
\min\Big\{(d-1+\theta_i)(\frac{1}{2}-\frac{1}{r_i}),
1\Big\},~~~i=1,2
   \\
&\rho_1+(d+\theta_1)(\frac{1}{2}-\frac{1}{r_1})-\frac{1}{q_1} =\mu,
\\&\rho_2+(d+\theta_2)(\frac{1}{2}-\frac{1}{r_2})-\frac{1}{q_2}
=1-\mu.
\endaligned\right.
\end{equation}
Then, for $f \in H^\mu$, we have
\begin{align}\label{str1}
\big\| U(\cdot) f\big\|_{L^{q_1}\big(\mathbb{R}; B^{\rho_1}_{r_1, 2}
\big)} &\leq C \|f\|_{H^\mu};\\\label{str2}\big\| K\ast f
\big\|_{L^{q_1}\big(I; B^{\rho_1}_{r_1, 2} \big)} &\leq C\big\|  f
\big\|_{L^{q_2'}\big(I; B^{-\rho_2}_{r'_2, 2} \big)};\\\label{str3}
\big\| K_{R}\ast f \big\|_{L^{q_1}\big(I; B^{\rho_1}_{r_1,  2}
\big)} &\leq C\big\|  f \big\|_{L^{q_2'}\big(I; B^{-\rho_2}_{r'_2,
2} \big)}.
\end{align}
where the subscript $R$ stands for retarded, and
\begin{align*}
K*f&=\int_{\mathbb{R}}K(t-s)f(u(s))ds,\\
K_R\ast f&=\int_{0}^tK(t-s)f(u(s))ds.
\end{align*}

\end{lemma}

In addition to the $W$-norm defined in \eqref{scattersize}, we also
need the following space
\begin{equation}
[W]^*(I)=L_t^{\frac{2(d+1)}{d+3}}\big(I;B^{\frac12}_{\frac{2(d+1)}{d+3},2}(\mathbb{R}^d)\big).
\end{equation}

Now we give a nonlinear estimate which will be applied  to show the
small data scattering.
\begin{lemma}\label{full}
We have
\begin{align}\label{fullst2}
&\Big\|\big(V(\cdot)*|u|^2\big)v\Big\|_{[W]^*(I)}+\Big\|\big(V(\cdot)*(uv)\big)u\Big\|_{[W]^*(I)}\\\nonumber
\leq C &\big\| v \big\|_{[W](I)}\|u\|_{L_t^\infty(I;
\dot{H}^{1}_x)}^{\frac{2(d-3)}{d-1}}\|u\|_{[W](I)}
^{\frac{4}{d-1}}+C\big\| u
\big\|_{[W](I)}^{1+\frac{2}{d-1}}\|u\|_{L_t^\infty(I;
\dot{H}_x^1)}^{\frac{d-3}{d-1}}\|v\|_{L_t^\infty(I;
\dot{H}^{1}_x)}^{\frac{d-3}{d-1}}\|v\|_{[W](I)} ^{\frac{2}{d-1}}.
\end{align}
In particular,
\begin{equation}\label{nonlin}
\|(V(\cdot)*|u|^2)u\|_{[W]^*(I)}\leq
C\|u\|_{[W](I)}^{1+\frac{4}{d-1}}\|u\|_{L^\infty(I;H^1)}^{\frac{2(d-3)}{d-1}}.
\end{equation}
\end{lemma}

\begin{proof}
We only need to prove the estimate
$\|(V(\cdot)*|u|^2)v\|_{[W]^*(I)}$, since the estimate
$\|(V(\cdot)*(uv))u\|_{[W]^*(I)}$ is similar. From
  the Sobolve embedding: $W^{s,p}(\R^d)\hookrightarrow B^s_{p,2}(\R^d),~p\leq2;
 B^s_{q,2}(\R^d)\hookrightarrow  W^{s,q}(\R^d),~q\geq2,$ the fractional Leibnitz rule \cite{KV1}, and the H\"{o}lder and the
Young inequalities, we have
\begin{align}\nonumber
&\big\|(V*|u|^2)v\big\|_{L^{ q'}\big(I; B^{1/2}_{ {r}', 2}
\big)}\\\label{Bony}\nonumber \lesssim& \big\| V \big\|_{L^p_\ast}
\big\| v \big\|_{L^{q}\big(I; B^{1/2}_{r, 2} \big)} \big\|u
\big\|^2_{L^k(I; L^s)}+ \big\| V \big\|_{L^p_\ast} \big\| u
\big\|_{L^{q}\big(I; B^{1/2}_{r, 2} \big)} \big\|u \big\|_{L^k(I;
L^s )}\big\|v \big\|_{L^k(I; L^s )},
\end{align}
where  the exponents satisfy
\begin{equation}\label{er}
\left\{ \aligned \frac{d}{p} & = 2 \delta(r) + 2\delta(s), \\
\frac{2}{q}& + \frac{2}{k} =1.
\endaligned \right.
\end{equation}
Since $V(x)=|x|^{-4}\in L_\ast^\frac{d}4,$ if we take admissible
pair $q=r=\frac{2(d+1)}{d-1}$ and $\delta(s)=1+\frac1{k}$ (then
$\delta(r)=\frac{d}{d+1},k=d+1$), then
\begin{align}\label{eq1.4}
\|(V*|u|^2)v\|_{L^{q^\prime}(I;B^{\frac12}_{q^\prime,2})}\lesssim\big\|
v \big\|_{[W](I)} \big\|u \big\|^2_{L^{k}(I; L^{s} )}+ \big\| u
\big\|_{[W](I)} \big\|u \big\|_{L^{k}(I; L^{s} )}\big\|v
\big\|_{L^{k}(I; L^{s} )}.
\end{align}
The H\"older inequality and the Sobolev embedding theorem yield that
\begin{equation}\label{eq1.5}\|v \big\|_{L^{k}(I; L^{s}
)}\leq\|v\|_{L_t^\infty
L^{2^*}_x}^{\frac{d-3}{d-1}}\|v\|_{L_t^{\frac{2(d+1)}{d-1}}L_x^{\frac{2d(d+1)}{d^2-2d-1}}}
^{\frac{2}{d-1}}\lesssim\|v\|_{L_t^\infty
\dot{H}^{1}_x}^{\frac{d-3}{d-1}}\|v\|_{[W](I)} ^{\frac{2}{d-1}}.
\end{equation}
Plugging \eqref{eq1.5} into \eqref{eq1.4}, we get
\begin{equation*}\|(V*|u|^2)v\|_{L^{q^\prime}(I;B^{\frac12}_{q^\prime,2})}\lesssim\big\| v \big\|_{[W](I)}\|u\|_{L_t^\infty
\dot{H}^{1}_x}^{\frac{2(d-3)}{d-1}}\|u\|_{[W](I)}
^{\frac{4}{d-1}}+\big\| u
\big\|_{[W](I)}^{1+\frac{2}{d-1}}\|u\|_{L_t^\infty
\dot{H}_x^1}^{\frac{d-3}{d-1}}\|v\|_{L_t^\infty
\dot{H}^{1}_x}^{\frac{d-3}{d-1}}\|v\|_{[W](I)} ^{\frac{2}{d-1}}.
\end{equation*}

Thus we complete the proof of Lemma \ref{full}.
\end{proof}

Now, we can state the local well-posedness for $(\ref{equ1})$ with
large initial data and small data scattering in the energy space
$H^1\times L^2$.

\begin{theorem}[small data scattering]\label{small}
Assume $d\geq 5,$  and $(u_0,u_1)\in H^1(\mathbb{R}^d)\times
L^2(\mathbb{R}^d)$. There exists a small constant $\delta=\delta(E)$
such that if $\|(u_0,u_1)\|_{ H^1\times L^2}\leq E$ and $I$ is an
interval such that
$$\|\dot{K}(t)u_0 + K(t)u_1\|_{W(I)}\leq
\delta,$$ then there exists a unique strong solution $u$ to
\eqref{equ1} in $I\times \mathbb{R}^d$, with $u\in C(I;H^1)\cap
C^1(I;L^2)$ and
\begin{equation}\label{smalll}
\|u\|_{W(I)}\leq 2C\delta. \end{equation} Let
$(T_-(u_0,u_1),T_+(u_0,u_1))$ be the maximal time interval on which
$u$ is well-defined.
\end{theorem}
\begin{remark}\label{rem2.3}
(1) There exists $\tilde{\delta}$ such that if
$\|(u_0,u_1)\|_{{H}^1\times L^2}\leq\tilde{\delta},$ the conclusion
of Theorem \ref{small} applies to any interval $I.$ Indeed, by
Strichartz estimates, $\|\dot{K}(t)u_0 + K(t)u_1\|_{W(I)}\leq
C\tilde{\delta}$ and the claim follows.

(2) Given $(u_0,u_1)\in H^1\times L^2,$ there exists $(0\in) I$ such
that the hypothesis of Theorem \ref{small} is verified on $I$. This
is clear because, by Strichartz estimates, $\|\dot{K}(t)u_0 +
K(t)u_1\|_{W(\R)}<\infty.$
\end{remark}

Finally, we conclude this subsection by recalling the following
standard finite blow-up criterion.
\begin{lemma}[Standard finite blow-up criterion]\label{criterion}
If  $T_+(u_0,u_1)<+\infty,$ then
$$\|u\|_{W\big([0,T_+(u_0,u_1))\big)}=+\infty.$$
A corresponding result holds for $T_-(u_0,u_1)$.
\end{lemma}
The proof is similar to the one in Lemma 2.11 of \cite{KM}.

\subsection{Perturbation lemma}
In this part, we give the perturbation theory of the solution of
\eqref{equ1} with the global space-time estimate. First we recall
some notations in \cite{IMN}.

 With any real-valued function
$u(t,x)$, we associate the complex-valued function $\vec{u}(t,x)$ by
\begin{equation}
\vec{u}=\langle\nabla\rangle u-i\dot{u},\quad
u=\Re\langle\nabla\rangle^{-1}\vec{u}.
\end{equation}
Then the free and nonlinear Klein-Gordon equations are given by
\begin{align}
\begin{cases}
(\Box+1)u=0\Longleftrightarrow(i\partial_t+\langle\nabla\rangle)\vec{u}=0,\\
(\Box+1)u=-f(u)\Longleftrightarrow(i\partial_t+\langle\nabla\rangle)\vec{u}=-f(\langle\nabla\rangle^{-1}\Re\vec{u}),
\end{cases}
\end{align}
and the energy are written as
$$\tilde{E}(\vec{u})=E(u,\dot{u})=\frac12 \int_{\mathbb{R}^d} \big(\big|\dot{u}
\big|^2 +\big| \nabla u \big|^2 + \big| u \big|^2 \big)dx +
\frac{1}{4} \iint_{\mathbb{R}^d\times\mathbb{R}^d}
 \frac{|
u(t,x)|^2 |u(t,y)|^2}{|x-y|^4} dxdy.$$

\begin{lemma}\label{long}  Let $I$ be a
time interval, $t_0\in I$ and $\vec{u},\vec{w}\in C(I;L^2(\R^d))$
satisfy
\begin{align*}
(i\partial_t+\langle\nabla\rangle)\vec{u}=&-f(u)+eq(u)\\
(i\partial_t+\langle\nabla\rangle)\vec{w}=&-f(w)+eq(w).
\end{align*}
for some function $eq(u),eq(w)$. Assume that for some constants
$M,E>0$, we have
\begin{align}\label{eq2.20}
\|w\|_{ST(I)}\leq M,\\\label{equ2.201}  \|\vec{u}\|_{L_t^\infty
L^2_x(I\times \R^d)}+\|\vec{w}\|_{L_t^\infty L^2_x(I\times
\R^d)}\leq E,
\end{align}
 Let $t_0\in I$, and let $(u(t_0),u_t(t_0))$ be
close to $(w(t_0),w_t(t_0))$ in the sense that
\begin{equation}\label{eq2.21}
\big\|\big(u(t_0)-w(t_0),u_t(t_0)-w_t(t_0)\big)\big\|_{H^1\times
L^2}\leq E'.
\end{equation}
Let
$\vec{\gamma}_0=e^{i\langle\nabla\rangle(t-t_0)}(\vec{u}-\vec{w})(t_0)$
and assume also that we have smallness conditions
\begin{equation} \label{eq2.22}
\|\gamma_0\|_{ST(I)}+\|(eq(u),eq(w))\|_{ST^*(I)}\leq\epsilon,
\end{equation}
where $0<\epsilon<\epsilon_1=\epsilon_1( M, E)$ is a small constant
and $$ST^*(I)=[W]^\ast(I)+L_t^1(I;L_x^2(\R^d)).$$ Then we conclude
that
\begin{equation}\label{eq2.23}
\begin{aligned}
\|u-w\|_{ST(I)}\leq & C(M,E)\epsilon,\\
\|u\|_{ST(I)}\leq & C(M,E,E').
\end{aligned}
\end{equation}
\end{lemma}
\begin{proof}
Since $\|w\|_{ST(I)}\leq M$, there exists a partition of the right
half of $I$ at $t_0$:
$$t_0<t_1<\cdots<t_N,~I_j=(t_j,t_{j+1}),~I\cap(t_0,\infty)=(t_0,t_N),$$
such that $N\leq C(L,\delta)$ and for any $j=0,1,\cdots,N-1,$ we
have
\begin{equation}\label{ome}
\|w\|_{ST(I_j)}\leq\delta\ll1.
\end{equation}
The estimate on the left half of $I$ at $t_0$ is analogue, we omit
it.

Let
\begin{equation}
\gamma(t)=u(t)-w(t),~\vec{\gamma}_j(t)=e^{i\langle\nabla\rangle(t-t_j)}\vec
\gamma(t_j),~0\leq j\leq N-1,
\end{equation}
then $\gamma$ satisfies the following difference equation
\begin{align*}
\begin{cases}
(i\partial_t+\langle\nabla\rangle)\vec{\gamma}=&(V\ast|w|^2)\gamma+2\big[V\ast(\gamma
w)\big]+2\big[V\ast(\gamma
w)\big]\gamma\\&+(V\ast|\gamma|^2)w+(V\ast|\gamma|^2)\gamma+eq(u)-eq(w)\\
\vec{\gamma}(t_j)=\vec{\gamma}_j(t_j),
\end{cases}
\end{align*}
which implies that
\begin{align*}
\vec{\gamma}(t)=\vec{\gamma}_j(t)-i\int_{t_j}^te^{i\langle\nabla\rangle(t-s)}\Big(&(V\ast|w|^2)\gamma+2\big[V\ast(\gamma
w)\big]w+2\big[V\ast(\gamma
w)\big]\gamma\\&+(V\ast|\gamma|^2)w+(V\ast|\gamma|^2)\gamma+eq(u)-eq(w)\Big)ds,\\
\vec{\gamma}_{j+1}(t)=\vec{\gamma}_j(t)-i\int_{t_j}^{t_{j+1}}e^{i\langle\nabla\rangle(t-s)}\Big(&(V\ast|w|^2)\gamma+2\big[V\ast(\gamma
w)\big]w+2\big[V\ast(\gamma
w)\big]\gamma\\&+(V\ast|\gamma|^2)w+(V\ast|\gamma|^2)\gamma+eq(u)-eq(w)\Big)ds.
\end{align*}
By Lemma \ref{lem22} and Lemma \ref{full}, we have
\begin{align}\label{equ3}
&\|\gamma-\gamma_j\|_{ST(I_j)}+\|\gamma_{j+1}-\gamma_j\|_{ST(\R)}\\\nonumber
\lesssim&\big\|(V\ast|w|^2)\gamma+2\big[V\ast(\gamma
w)\big]w+2\big[V\ast(\gamma
w)\big]\gamma+(V\ast|\gamma|^2)w+(V\ast|\gamma|^2)\gamma\big\|_{[W]^\ast(I_j)}\\\nonumber&+\|(eq(u),eq(w))\|_{ST^\ast(I_j)}\\\nonumber
\lesssim &\|\gamma \|_{[W](I_j)}\|w\|_{L_t^\infty(I_j;
\dot{H}^{1}_x)}^{\frac{2(d-3)}{d-1}}\|w\|_{[W](I_j)}
^{\frac{4}{d-1}}+\big\| w
\big\|_{[W](I_j)}^{1+\frac{2}{d-1}}\|w\|_{L_t^\infty(I_j;
\dot{H}_x^1)}^{\frac{d-3}{d-1}}\|\gamma\|_{L_t^\infty(I_j;
\dot{H}^{1}_x)}^{\frac{d-3}{d-1}}\|\gamma\|_{[W](I_j)}
^{\frac{2}{d-1}}\\\nonumber &+\|w
\|_{[W](I_j)}\|\gamma\|_{L_t^\infty(I_j;
\dot{H}^{1}_x)}^{\frac{2(d-3)}{d-1}}\|\gamma\|_{[W](I_j)}
^{\frac{4}{d-1}}+\big\| \gamma
\big\|_{[W](I_j)}^{1+\frac{2}{d-1}}\|\gamma\|_{L_t^\infty(I_j;
\dot{H}_x^1)}^{\frac{d-3}{d-1}}\|w\|_{L_t^\infty(I_j;
\dot{H}^{1}_x)}^{\frac{d-3}{d-1}}\|w\|_{[W](I_j)}
^{\frac{2}{d-1}}\\\nonumber
&+\|\gamma\|_{[W](I_j)}^{1+\frac{4}{d-1}}\|\gamma\|_{L^\infty(I_j;H^1)}^{\frac{2(d-3)}{d-1}}+\|(eq(u),eq(w))\|_{ST^\ast(I_j)}.
\end{align}
Therefore, assuming that
\begin{equation}\label{laodong}
\|\gamma\|_{ST(I_j)}\leq\delta\ll1,~\forall~j=0,1,\cdots,N-1,
\end{equation}
then by \eqref{ome} and \eqref{equ3}, we have
\begin{equation}
\|\gamma\|_{ST(I_j)}+\|\gamma_{j+1}\|_{ST(t_{j+1},t_N)}\leq
C\|\gamma_j\|_{ST(t_j,t_N)}+\epsilon,
\end{equation}
for some absolute constant $C>0$. By \eqref{eq2.22} and iteration on
$j$, we obtain
\begin{equation}
\|\gamma\|_{ST(I)}\leq (2C)^N\epsilon\leq\frac{\delta}2,
\end{equation}
if we choose $\epsilon_1$ sufficiently small. Hence the assumption
\eqref{laodong} is justified by continuity in $t$ and induction on
$j$. Then repeating the estimate \eqref{equ3} once again, we can get
the ST-norm estimate on $\gamma$, which implies the Strichartz
estimates on $u$.
\end{proof}

\section{Profile decomposition}
\setcounter{section}{3}\setcounter{equation}{0}

In this section, we first recall the linear profile decomposition of
the sequence of $H^1$-bounded solutions of \eqref{equ1} which was
established in \cite{IMN}. And then we utilize it to show the
orthogonal analysis for the nonlinear energy and the nonlinear
profile decomposition which will be used to construct the critical
element and obtain its compactness properties.

 \subsection{Linear profile decomposition}
First, we give some notation as introduced in \cite{IMN}. For any
triple $(t_n^j,x_n^j,h_n^j)\in\R\times\R^d\times(0,\infty)$ with
arbitrary suffix $n$ and $j$, let $\tau_n^j,~T_n^j$, and
$\langle\nabla\rangle_n^j$ respectively denote the scaled time
shift, the unitary and the self-adjoint operators in $L^2(\R^d)$,
defined by
\begin{equation}\label{equ6.1}
\tau_n^j=-\frac{t_n^j}{h_n^j},~T_n^j\varphi(x)=(h_n^j)^{-\frac{d}{2}}\varphi\big(\frac{x-x_n^j}{h_n^j}\big),~\langle\nabla\rangle_n^j=\sqrt{-\Delta
+(h_n^j)^2}.
\end{equation}
 We denote the set of Fourier multipliers on
$$\mathcal{MC}=\big\{\mu=\mathcal{F}^{-1}\tilde{\mu}\mathcal{F}|\ \tilde{\mu}\in
C(\mathbb{R}^d),\exists
\lim\limits_{|x|\rightarrow\infty}\tilde{\mu}(x)\in\mathbb{R}\big\}.$$

Now we can state the linear profile decomposition as follows
\begin{lemma}[Linear profile decomposition,
\cite{IMN}]\label{lem3.1} Let $\vec{v}_n(t)=e^{i\langle\nabla\rangle
t}\vec{v}_n(0)$ be a sequence of free Klein-Gordon solutions with
uniformly bounded $L^2_x$ norm. Then after replacing it with some
subsequence, there exist $K\in\{0,1,2\ldots,\infty\}$ and, for each
integer $j\in[0,K)$, $\varphi^j\in L^2(\mathbb{R}^d)$ and $\{(t^j_n,
x^j_n,h_n^j)\}_{n\in\mathbb{N}}\subset\mathbb{R}\times\mathbb{R}^d\times(0,1]$
satisfying the following. Define $\vec{v}^j_n$ and
$\vec{\omega}^k_n$ for each $j<k\leq K$ by
\begin{equation}\begin{split}\label{equ3.1}
\vec{v}_n(t,x)=\sum\limits_{j=0}^{k-1}\vec{v}^j_n(t,x)+\vec{\omega}^k_n(t,x),\\
\vec{v}^j_n(t,x)=e^{i\langle\nabla\rangle(t-t^j_n)}T_n^j\varphi^j(x)=T_n^j\big(e^{i\langle\nabla\rangle_n^j\frac{t-t_n^j}{h_n^j}}\varphi^j\big)
,\end{split}
\end{equation}
then  we have
\begin{equation}\label{equ3.2}
\lim\limits_{k\rightarrow
K}\varlimsup\limits_{n\rightarrow\infty}\|\vec{\omega}^k_n\|_{L^\infty(\mathbb{R};B^{-\frac{d}{2}}_{\infty,\infty}(
\mathbb{R}^d))}=0,
\end{equation}
and for any $\mu\in \mathcal{MC}$, any $l<j<k\leq K$ and any $t\in
\mathbb{R}$,
\begin{align}\label{equ3.3}
\lim\limits_{n\rightarrow\infty}\langle\mu\vec{v}^l_n,
\mu\vec{v}^j_n\rangle_{L^2_x}^2=0=
\lim\limits_{n\rightarrow\infty}\langle\mu\vec{v}^j_n,
\mu\vec{\omega}^k_n\rangle_{L^2_x}^2,
\\\label{equ3.4}
\lim\limits_{n\rightarrow\infty}\Big|\frac{h_n^l}{h_n^j}\Big|+\Big|\frac{h_n^j}{h_n^l}\Big|+\frac{|t_n^j-t_n^k|+|x_n^j-x_n^k|}{h_n^l}=+\infty.
\end{align}
Moreover, each sequence $\{h_n^j\}_{n\in\mathbb{N}}$ is either going
to $0$ or identically $1$ for all $n$.
\end{lemma}

\begin{remark}
We call $\{\vec{v}_n^j\}_{n\in\mathbb{N}}$ a free concentrating wave
for each $j$, and $\vec{w}_n^k$ the remainder. From \eqref{equ3.3},
we have the following asymptotic orthogonality
\begin{equation}\label{orth}
\lim\limits_{n\rightarrow+\infty}\Big(\|\mu\vec{v}_n(t)\|_{L^2}^2-\sum\limits_{j=0}^{k-1}\|\mu\vec{v}_n^j(t)\|_{L^2}^2
-\|\mu\vec{\omega}_n^k(t)\|_{L^2}^2\Big)=0,\quad \forall~\mu\in
\mathcal{MC}.
\end{equation}
\end{remark}
Next we begin with the orthogonal analysis for the nonlinear energy.
It follows from Mikhlin's theorem that the following estimates for
$1<p<\infty,$
\begin{align}\label{equ1.6.1}
\big\|\big[|\nabla|-\langle\nabla\rangle_n\big]\varphi\big\|_p\lesssim&
h_n\big\|\langle\nabla/h_n\rangle^{-1}\varphi\big\|_p,\\\label{equ1.6.2}
\big\|\big[|\nabla|^{-1}-\langle\nabla\rangle_n^{-1}\big]\varphi\big\|_p\lesssim&
\big\|\langle\nabla/h_n\rangle^{-2}|\nabla|^{-1}\varphi\big\|_p,
\end{align}
hold uniformly for $0<h_n\leq1.$
\begin{lemma}\label{energy}
Let $\vec{v}_n$ be a sequence of free Klein-Gordon solutions
satisfying $\vec{v}_n(0)\in L^2_x$. Let
$\vec{v}_n=\sum\limits_{j=0}^{k-1}\vec{v}^j_n+\vec{\omega}_n^k$ be
the linear profile decomposition given by Lemma \ref{lem3.1}. If
$\varlimsup\limits_{n\rightarrow\infty}\tilde{E}(\vec{v}_n(0))<+\infty$,
then we have $\vec{v}_n^j(0)\in L^2_x$ for large $n$, and
\begin{equation}\label{equ3.5}
\lim\limits_{k\rightarrow
K}\varlimsup\limits_{n\rightarrow\infty}\Big|\tilde{E}(\vec{v}_n(0))-\sum\limits_{j=0}^{k-1}
\tilde{E}(\vec{v}^j_n(0))-\tilde{E}(\vec{\omega}^k_n(0))\Big|=0.
\end{equation}
Moreover, we have for all $j<k$
\begin{equation}\label{equ3.6}
0\leq\varliminf\limits_{n\rightarrow\infty}\tilde{E}(\vec{v}^j_n(0))\leq\varlimsup\limits_{n\rightarrow
\infty}\tilde{E}(\vec{v}^j_n(0))\leq\varlimsup\limits_{n\rightarrow
\infty}\tilde{E}(\vec{v}_n(0)),
\end{equation}
where the last inequality becomes equality only if $K=1$ and
$\vec{\omega}^1_n\rightarrow 0$ in $L^\infty_t L^2_x$.

\end{lemma}

\begin{proof}
First, we claim that
\begin{equation}\label{claim}
\|u\|_{L_x^{2^\ast}}\lesssim\|u\|_{H^1}^\frac{d-2}d\|u\|_{B^{1-\frac{d}2}_{\infty,\infty}}^\frac2d,\quad
2^\ast=\frac{2d}{d-2}.
\end{equation}
In fact, on one hand, by the H\"older and Bernstein equalities, we
have
$$\|P_{\leq 1}u\|_{L_x^{2^\ast}}\lesssim\|P_{\leq 1}u\|_{L_x^2}^\frac{d-2}d\|P_{\leq
1}u\|_{L_x^\infty}^\frac2d\lesssim\|u\|_{H^1}^\frac{d-2}d\|u\|_{B^{1-\frac{d}2}_{\infty,\infty}}^\frac2d,$$
On the other hand, from the sharp interpolation \cite{BCD}, we know
$$\|P_{>1}u\|_{L_x^{2^\ast}}\lesssim\|P_{>1}u\|_{L_x^2}^\frac{d-2}d\|P_{>
1}u\|_{\dot{B}^{1-\frac{d}2}_{\infty,\infty}}^\frac2d\lesssim\|u\|_{H^1}^\frac{d-2}d\|u\|_{B^{1-\frac{d}2}_{\infty,\infty}}^\frac2d,$$
which concludes the claim.

Thus, by \eqref{claim} and \eqref{equ3.2}, we obtain
$$\lim\limits_{k\rightarrow
K}\varlimsup\limits_{n\rightarrow\infty}\|\omega^k_n\|_{L_x^{2^\ast}}\leq\lim\limits_{k\rightarrow
K}\varlimsup\limits_{n\rightarrow\infty}\|\omega^k_n\|_{H^1}^\frac{d-2}d\|\omega^k_n\|_{B^{1-\frac{d}2}_{\infty,\infty}}^\frac2d=0,$$
where  $\omega^k_n=\Re \langle\nabla\rangle^{-1}\vec{\omega}^k_n$.
This implies that, if there exists $i\in\{1,2,3,4\}$ such that
$u_i=\omega_n^k$, then by the H\"{o}lder and the
Hardy-Littlewood-sobolev inequalities, we get
\begin{align*}
\lim\limits_{k\rightarrow
K}\varlimsup\limits_{n\rightarrow\infty}\|\big(V(x)*(u_1u_2)\big)(u_3u_4)\|_{L^1_x}\leq
\lim\limits_{k\rightarrow
K}\varlimsup\limits_{n\rightarrow\infty}\prod\limits_{i=1}^{4}\|u_i\|_{L_x^{2^\ast}}
=0.
\end{align*}
This together with \eqref{orth} reduces us to prove
\begin{equation}\label{reduce}
\lim\limits_{k\rightarrow
K}\varlimsup\limits_{n\rightarrow\infty}\Big|F\big(\sum_{j<k}v_n^j(0)\big)-\sum_{j<k}F\big(v_n^j(0)\big)\Big|=0,\end{equation}
where $F(u)=\big\|(V(x)*|u|^2)|u|^2\big\|_{L_x^1}.$

Moreover, using the decay of $e^{it\langle\nabla\rangle}$ in
$\mathcal{S}\rightarrow L^{2^\ast}_x$ uniform w.r.t. $n$ and the
Sobolev embedding $\dot{H}^1(\R^d)\subset L^{2^\ast}(\R^d)$, we have
$$\|v_n^j\|_{L_x^{2^\ast}}\leq\|\langle\nabla\rangle^{-1}e^{-i\langle\nabla\rangle_n^j\tau^j_n}\varphi^j(x)\|_{L_x^{2^\ast}}\rightarrow
0,\ \text{as}\ n\rightarrow\infty.$$ Thus, we can discard those $j$
where $\tau_n^j=-\frac{t_n^j}{h_n^j}\rightarrow+\infty.$

Hence, up to subsequence, we may assume that $\tau_n^j\to\exists
\tau_\infty^j\in\R$ for all $j$. Let
\begin{equation}
\psi^j:=\Re
e^{-i\langle\nabla\rangle_\infty^j\tau_\infty^j}\varphi^j\in
L_x^2(\R^d),
\end{equation}
we have
\begin{align}\nonumber
&\Big|F\big(\sum_{j<k}v_n^j(0)\big)-\sum_{j<k}F\big(v_n^j(0)\big)\Big|\\
\leq&\Big|F\big(\sum_{j<k}v_n^j(0)\big)-F\big(\sum_{j<k}\langle\nabla\rangle^{-1}T_n^j\psi^j\big)\Big|\\
&+\Big|\sum_{j<k}F\big(v_n^j(0)\big)-\sum_{j<k}F\big(\langle\nabla\rangle^{-1}T_n^j\psi^j\big)\Big|\\\label{equ3.20}
&+\Big|F\big(\sum_{j<k}\langle\nabla\rangle^{-1}T_n^j\psi^j\big)-\sum_{j<k}F\big(\langle\nabla\rangle^{-1}T_n^j\psi^j\big)\Big|.
\end{align}
By the continuity of the operator $e^{it\langle\nabla\rangle}$ in
$t$ in $H^1$, we have
$$v_n^j(0)-\langle\nabla\rangle^{-1}T_n^j\psi^j\to0~\text{in}~H^1(\R^d),~\text{as}~n\to\infty.
$$
This together with the following nonlinear estimate
\begin{equation}\label{non}
\big\|\big(V(\cdot)\ast(g_1g_2)\big)g_3g_4\big\|_{L_x^1}\lesssim\prod\limits_{j=1}^4\|g_j\|_{L_x^{2^\ast}}\end{equation}
show that as $n\to\infty,$
\begin{align*}
\Big|F\big(\sum_{j<k}v_n^j(0)\big)-F\big(\sum_{j<k}\langle\nabla\rangle^{-1}T_n^j\psi^j\big)\Big|\to0,\\
\Big|\sum_{j<k}F\big(v_n^j(0)\big)-\sum_{j<k}F\big(\langle\nabla\rangle^{-1}T_n^j\psi^j\big)\Big|\to0.
\end{align*}

Now we consider the term \eqref{equ3.20}. Let
\begin{align*}
\hat{\psi}^j= \begin{cases}|\nabla|^{-1}\psi^j,\quad
\text{if}~h_n^j\to0\\
\langle\nabla\rangle^{-1}\psi^j,\quad \text{if}~h_n^j\equiv1,
\end{cases}
\end{align*}
then we have $\hat{\psi}^j\in L^{2^\ast}_x$, and
\begin{align}\nonumber
&\Big|F\big(\sum_{j<k}\langle\nabla\rangle^{-1}T_n^j\psi^j\big)-\sum_{j<k}F\big(\langle\nabla\rangle^{-1}T_n^j\psi^j\big)\Big|\\
\lesssim&\Big|F\big(\sum_{j<k}\langle\nabla\rangle^{-1}T_n^j\psi^j\big)-F\big(\sum_{j<k}h_n^jT_n^j\hat{\psi}^j\big)\Big|\\
&+\Big|\sum_{j<k}F\big(\langle\nabla\rangle^{-1}T_n^j\psi^j\big)-\sum_{j<k}F\big(h_n^jT_n^j\hat{\psi}^j\big)\Big|\\
&+\Big|F\big(\sum_{j<k}h_n^jT_n^j\hat{\psi}^j\big)-\sum_{j<k}F\big(h_n^jT_n^j\hat{\psi}^j\big)\Big|.
\end{align}
By \eqref{equ1.6.1}, one has
\begin{align*}
\big\|\langle\nabla\rangle^{-1}T_n^j\psi^j-h_n^jT_n^j\hat{\psi}^j\big\|_{L_x^{2^\ast}}=&\begin{cases}
\big\|\langle\nabla\rangle^{-1}T_n^j\psi^j-h_n^jT_n^j|\nabla|^{-1}\psi^j\big\|_{L_x^{2^\ast}}~~\text{if}~~h_n^j\to0\\
\big\|\langle\nabla\rangle^{-1}T_n^j\psi^j-h_n^jT_n^j\langle\nabla\rangle^{-1}\psi^j\big\|_{L_x^{2^\ast}}~~\text{if}~~h_n^j\equiv1
\end{cases}\\
=&\begin{cases} \big\|(\langle\nabla\rangle_n^j)^{-1}\psi^j-|\nabla|^{-1}\psi^j\big\|_{L_x^{2^\ast}}~~\text{if}~~h_n^j\to0\\
0~~\text{if}~~h_n^j\equiv1
\end{cases}\\
\rightarrow&\quad 0,~\text{as}~~n\to\infty.
\end{align*}
Combining this with \eqref{non}, we obtain  that as $n\to\infty,$
\begin{align*}
\Big|F\big(\sum_{j<k}\langle\nabla\rangle^{-1}T_n^j\psi^j\big)-F\big(\sum_{j<k}h_n^jT_n^j\hat{\psi}^j\big)\Big|\to0,\\
\Big|\sum_{j<k}F\big(\langle\nabla\rangle^{-1}T_n^j\psi^j\big)-\sum_{j<k}F\big(h_n^jT_n^j\hat{\psi}^j\big)\Big|\to0.
\end{align*}
Thus it suffices to show that as $n\to\infty$
\begin{equation}\label{guijie}
\Big|F\big(\sum_{j<k}h_n^jT_n^j\hat{\psi}^j\big)-\sum_{j<k}F\big(h_n^jT_n^j\hat{\psi}^j\big)\Big|\rightarrow
0.
\end{equation}
Now we define $\hat{\psi}^j_{n,R}$ for any $R\gg1$ by
$$\hat{\psi}^j_{n,R}(x)=\chi_R(x)\hat{\psi}^j\prod\big\{(1-\chi_{h_n^{j,l}R})(x-x_{n}^{j,l})~\big|~1\leq
l<k,~h_n^lR<h_n^j\big\},$$ where
$(h_n^{j,l},x_n^{j,l})=(h_n^l,x_n^j-x_n^l)/h_n^j,$ and
$\chi_R(x)=\chi(\frac{x}{R})$ with $\chi(x)\in
C^\infty_c(\mathbb{R}^{d})$ satisfing $\chi(x)=1$ for $|x|\leq 1$
and $\chi(x)=0$ for $|x|\geq 2$. Then
$\hat{\psi}_{n,R}^j\to\chi_R\hat{\psi}^j$ in $L_x^{2^\ast}$
 as $n\to\infty$, since either $h_n^{j,l}\to0$ or
$|x_n^{j,l}|\to\infty$ by \eqref{equ3.4}. Moreover, we have
$\chi_R\hat{\psi}^j\to\hat{\psi}^j$ in $L_x^{2^\ast}$
 as $R\to\infty$.

Hence we may replace $\hat{\psi}^j$ by $\hat{\psi}^j_{n,R}$ in
\eqref{guijie}. Since
$\big\{\text{supp}_{(t,x)}h_n^jT_n^j\hat{\psi}^j_{n,R}\big\}$ are
mutually disjoint for large $n$, and so for large $n$
\begin{equation}
\big|\sum_{j<k}h_n^jT_n^j\hat{\psi}^j_{n,R}\big|^2=\sum_{j<k}\big|h_n^jT_n^j\hat{\psi}^j_{n,R}\big|^2.
\end{equation}
Then
\begin{align*}
&\Big|F\big(\sum_{j<k}h_n^jT_n^j\hat{\psi}^j_{n,R}\big)-\sum_{j<k}F\big(h_n^jT_n^j\hat{\psi}^j_{n,R}\big)\Big|\\
\leq&\sum_{j\neq
l}\Big\|\big(V(\cdot)\ast|h_n^jT_n^j\hat{\psi}^j_{n,R}|^2\big)|h_n^lT_n^l\hat{\psi}^l_{n,R}|^2\Big\|_{L_x^1(\R^d)}\\
=&\sum_{j\neq
l}\big(h_{n}^{j,l}\big)^{2-d}\Big\|\big(V(\cdot)\ast|\hat{\psi}^j_{n,R}|^2\big)
|\hat{\psi}^l_{n,R}\big(\frac{x-x_n^{j,l}}{h_n^{j,l}}\big)|^2\Big\|_{L_x^1(\R^d)}\\
\to&\quad 0,~\text{as}~n\to\infty,
\end{align*}
by Lebesgue dominated convergence theorem, since either
$h_n^{j,l}\to0$ or $|x_n^{j,l}|\to\infty$ by \eqref{equ3.4}. This
concludes the proof of Lemma \ref{energy}.
\end{proof}

\subsection{Nonlinear profile decomposition}
After the linear profile decomposition of a sequence of initial data
in the last subsection, we now show the nonlinear profile
decomposition of a sequence of the solutions of \eqref{equ1} with
the same initial data in the energy space $H^1(\mathbb{R}^d)\times
L^2(\mathbb{R}^d)$  by following the argument in  \cite{IMN}.

First we construct a nonlinear profile corresponding to a free
concentrating wave. Let $\vec{v}_n$ be a free concentrating wave for
a sequence $(t_n,x_n,h_n)\in\R\times\R^d\times(0,1]$,
\begin{align}
\begin{cases}
(i\partial_t+\langle\nabla\rangle)\vec{v}_n=0,\\
\vec{v}_n(t_n)=T_n\phi(x),\ \phi(x)\in L^2(\R^d).
\end{cases}
\end{align}
Then by Lemma \ref{lem3.1}, we have a sequence of the free
concentrating wave $\vec{v}_n^j(t,x)$ with
$\vec{v}_n^j(t_n^j)=T_n^j\varphi^j,~\varphi^j\in L^2(\R^d)$ for
$j=0,1,\cdots,k-1,$ such that
\begin{align*}
\vec{v}_n(t,x)=&\sum_{j=0}^{k-1}\vec{v}^j_n(t,x)+\vec{\omega}^k_n(t,x)\\
=&\sum_{j=0}^{k-1}e^{i\langle\nabla\rangle(t-t^j_n)}T_n^j\varphi^j(x)+\vec{\omega}^k_n(t,x)\\
=&\sum_{j=0}^{k-1}T_n^je^{i\big(\frac{t-t^j_n}{h_n^j}\big)\langle\nabla\rangle_n^j}\varphi^j+\vec{\omega}^k_n(t,x).
\end{align*}

Now for any concentrating wave $\vec{v}_n^j$, we undo the group
action $T_n^j$ to look for the linear profile $\vec{V}^j$. Let
$$\vec{v}_n^j(t,x)=T_n^j\vec{V}_n^j\big((t-t_n^j)/h_n^j\big),$$
then we have
$$\vec{V}_n^j(t,x)=e^{it\langle\nabla\rangle_n^j}\varphi^j.$$

Now let $\vec{u}_n^j$ be the nonlinear solution with the same
initial data $\vec{v}_n^j(0)$
\begin{align}
\begin{cases}
(i\partial_t+\langle\nabla\rangle)\vec{u}_n^j=-f\big(\Re\langle\nabla\rangle^{-1}\vec{u}_n^j\big),\\
\vec{u}_n^j(0)=\vec{v}_n^j(0)=T_n^j\vec{V}^j_n(\tau_n^j),
\end{cases}
\end{align}
where $\tau_n^j=-t_n^j/h_n^j.$ In order to look for the nonlinear
profile $\vec{U}_\infty^j$ associate with the free concentrating
wave $\vec{v}_n^j$, we also need undo the group action. Define
$$\vec{u}_n^j(t,x)=T_n^j\vec{U}_n^j\big((t-t_n^j)/h_n^j\big),$$
then $\vec{U}^j_n$ satisfies the rescaled equation
\begin{align*}\begin{cases}
(i\partial_t+\langle\nabla\rangle_n^j)\vec{U}_n^j=-f\big(\Re(\langle\nabla\rangle_n^j)^{-1}\vec{U}_n^j\big),\\
\vec{U}_n^j(\tau_n^j)=\vec{V}^j_n(\tau_n^j).
\end{cases}
\end{align*}

Up to subsequence, we may assume that there exist
$h_\infty^j\in\{0,1\}$ and $\tau_\infty^j\in[-\infty,\infty]$ for
every $j$, such that as $n\to\infty$
$$h_n^j\rightarrow h^j_\infty,~\text{and}~\tau_n^j\to\tau_\infty^j.$$
And then the limit equations are given by
\begin{align*}\vec{V}_\infty^j=e^{it\langle\nabla\rangle_\infty^j}\varphi^j,\quad
\begin{cases}
(i\partial_t+\langle\nabla\rangle_\infty^j)\vec{U}_\infty^j=-f\big(\hat{U}_\infty^j\big),\\
\vec{U}_\infty^j(\tau_\infty^j)=\vec{V}^j_\infty(\tau_\infty^j),
\end{cases}
\end{align*}
where $\hat{U}_\infty^j$ is defined by
\begin{align}\label{un1}
\hat{U}_\infty^j:=\Re(\langle\nabla\rangle_\infty^j)^{-1}\vec{U}_\infty^j=
\begin{cases}
\Re\langle\nabla\rangle^{-1}\vec{U}_\infty^j~~\text{if}~~h^j_\infty=1,\\
\Re|\nabla|^{-1}\vec{U}_\infty^j~~\text{if}~~h^j_\infty=0.
\end{cases}
\end{align}

We remark that  by using the standard iteration with the Strichartz
estimate, we can obtain the unique existence of a local solution
$\vec{U}_\infty^j$ around $t=\tau_\infty^j$ in all cases, including
$h_\infty^j=0$ and $\tau_\infty^j=\pm\infty$ (the later
corresponding to the existence of the wave operators). We denote
$\vec{U}_\infty^j$ on the maximal existence interval to be the
nonlinear profile associated with the free concentrating wave
$(\vec{v}_n^j;t_n^j,x_n^j,h_n^j)$.

The nonlinear concentrating wave $\vec{u}_{(n)}^j$ associated with
$\vec{v}_n^j$ is defined by
\begin{equation}\label{un}
\vec{u}_{(n)}^j(t,x):=T_n^j\vec{U}_\infty^j\big((t-t_n^j)/h_n^j\big).
\end{equation}
It is easy to see that  $u_{(n)}^j$ solves \eqref{equ1} when
$h_\infty^j=1$. When $h_\infty^j=0,$   $u_{(n)}^j$ solves
\begin{align*}\begin{cases}
(\partial_{tt}-\Delta+1)u^j_{(n)}=(i\partial_t+\langle\nabla\rangle)\vec{u}_{(n)}^j
=\big(\langle\nabla\rangle-|\nabla|\big)\vec{u}_{(n)}^j-f\big(|\nabla|^{-1}\langle\nabla\rangle
u_{(n)}^j\big),\\
\vec{u}_{(n)}^j(0)=T_n^j\vec{U}_\infty^j(\tau_n^j).
\end{cases}
\end{align*}
The existence time interval of $u_{(n)}^j$ may be finite and even go
to $0$, but at least we have
\begin{equation}\label{lea}
\begin{split}
\|\vec{u}_n^j(0)&-\vec{u}_{(n)}^j(0)\|_{L_x^2}=\big\|T_n^j\vec{V}^j_n(\tau_n^j)-T_n^j\vec{U}_\infty^j(\tau_n^j)\big\|_{L_x^2}\\
\leq&\big\|\vec{V}^j_n(\tau_n^j)-\vec{V}_\infty^j(\tau_n^j)\big\|_{L_x^2}+\big\|\vec{V}^j_\infty(\tau_n^j)-\vec{U}_\infty^j(\tau_n^j)\big\|_{L_x^2}
\to0,
\end{split}
\end{equation}
as $n\to\infty.$

Let $u_n$ be a sequence of (local) solutions of \eqref{equ1} around
$t=0$, and let $v_n$ be the sequence of the free solutions with the
same initial data. We consider the linear profile decomposition of
$\{\vec{v}_n\}$ given by Lemma \ref{lem3.1},
\begin{equation*}
\vec{v}_n=\sum\limits_{j=0}^{k-1}\vec{v}^j_n+\vec{\omega}^k_n,\quad
\vec{v}^j_n=e^{i\langle\nabla\rangle(t-t^j_n)}T_n^j\varphi^j.
\end{equation*}

\begin{definition}
[Nonlinear profile decomposition] Let
$\{\vec{v}_n^j\}_{n\in\mathbb{N}}$ be the free concentrating wave,
and $\{\vec{u}_{(n)}^j\}_{n\in\mathbb{N}}$ be the sequence of the
nonlinear concentrating wave associated with
$\{\vec{v}_n^j\}_{n\in\mathbb{N}}$. Then we define the nonlinear
profile decomposition of $u_n$ by
\begin{equation}\label{nonlineard}
\vec{u}_{(n)}^{<k}:=\sum\limits_{j=0}^{k-1}\vec{u}_{(n)}^j=\sum_{j=0}^{k-1}T_n^j\vec{U}_\infty^j\big((t-t_n^j)/h_n^j\big).
\end{equation}
\end{definition}

We will show that $\vec{u}_{(n)}^{<k}+\vec{\omega}_n^k$ is a good
approximation for $\vec{u}_n$, provided that each nonlinear profile
has finite global Strichartz norm.

Next we introduce some Strichartz norms. Let $ST(I)$ and
$ST^\ast(I)$ be the functions spaces on $I\times\R^d$ defined as
above
\begin{align*}
ST(I)=&[W](I)=L_t^{\frac{2(d+1)}{d-1}}(I;B^{\frac12}_{\frac{2(d+1)}{d-1},2}(\mathbb{R}^d)),\\
ST^\ast(I)=&[W]^*(I)+L_t^1(I;L^2(\R^d)).
\end{align*}
The Strichartz norm for the nonlinear profile $\hat{U}_\infty^j$
depends on the scaling $h_\infty^j$
\begin{align}
ST_\infty^j(I):= \begin{cases} ST(I)~~\qquad\qquad\text{if}~~h_\infty^j=1,\\
L_t^q(I;\dot{B}^\frac12_{q,2})~~\big(q=\frac{2(d+1)}{d-1}\big)~~\text{if}~~h_\infty^j=0.
\end{cases}
\end{align}

The following two lemmas derive from Lemma \ref{lem3.1} and the
perturbation lemma. The first lemma concerns the orthogonality in
the Strichartz norms.
\begin{lemma}\label{lem3.3}
Assume that in the nonlinear profile decomposition
\eqref{nonlineard}, we have
\begin{equation}\label{equ3.7}
\|\hat{U}^j_\infty\|_{ST_\infty^j(\mathbb{R})}+\|\vec{U}^j_\infty\|_{L_t^\infty
L^2_x(\mathbb{R})}<\infty,\ \forall j<k.
\end{equation}
Then, for any finite interval $I, j<k,$ we have
\begin{align}\label{equ3.8}
\varlimsup\limits_{n\rightarrow\infty}\|u_{(n)}^j\|_{ST(I)}&\lesssim\|\hat{U}_\infty^j\|_{ST_\infty^j(\mathbb{R})},\\\label{equ3.9}
\varlimsup\limits_{n\rightarrow\infty}\|u_{(n)}^{<k}\|_{ST(I)}^2&\lesssim\varlimsup\limits_{n\rightarrow\infty}
\sum\limits_{j=0}^{k-1}\big\|u_{(n)}^j\big\|_{ST(\mathbb{R})}^2,
\end{align}
where the implicit constants do not depend on the interve $I$ or
$j$. We also have
\begin{equation}\label{equ3.10}
\lim\limits_{n\rightarrow\infty}\bigg\|f\big(u_{(n)}^{<k}\big)-\sum\limits_{j=0}^{k-1}
f\big((\langle\nabla\rangle_\infty^j)^{-1}\langle\nabla\rangle
u_{(n)}^j\big)\bigg\|_{ST^*(I)}=0,
\end{equation}
where $f(u)=(V(x)\ast|u|^2)u.$
\end{lemma}

\begin{proof}
One can refer to \cite{IMN} for the proof of \eqref{equ3.8} and
\eqref{equ3.9}. Now we turn to prove \eqref{equ3.10}. By the
definition of $u_{(n)}^j$ and $\hat{U}_\infty^j$, we know that
$$u_{(n)}^j(x,t)=\Re\langle\nabla\rangle^{-1}\vec{u}_{(n)}^j(t,x)=\Re\langle\nabla\rangle^{-1}T_n^j\vec{U}_\infty^j\Big(\frac{t-t_n^j}{h_n^j}\Big)
=h_n^jT_n^j\frac{\langle\nabla\rangle_\infty^j}{\langle\nabla\rangle_n^j}
\hat{U}_\infty^j\Big(\frac{t-t_n^j}{h_n^j}\Big).$$ Let $u_{\langle
n\rangle}^{<k}(t,x)=\sum\limits_{j<k}u_{\langle n\rangle}^{j}(x,t),$
where $u_{\langle n\rangle}^{j}(x,t)$ is defined by
$$u_{\langle
n\rangle}^{j}(x,t)=\frac{\langle\nabla\rangle}{\langle\nabla\rangle_\infty^j}u_{(n)}^j
=h_n^jT_n^j \hat{U}_\infty^j\Big(\frac{t-t_n^j}{h_n^j}\Big).$$ Then
we have
\begin{align}\nonumber
&\big\|f\big(u_{(n)}^{<k}\big)-\sum\limits_{j<k}
f\big((\langle\nabla\rangle_\infty^j)^{-1}\langle\nabla\rangle
u_{(n)}^j\big)\big\|_{ST^*(I)}\\\nonumber
=&\big\|f\big(u_{(n)}^{<k}\big)-\sum\limits_{j<k} f(u_{\langle
n\rangle}^{j})\big\|_{ST^*(I)}\\\label{equ3.26}
\leq&\big\|f\big(u_{(n)}^{<k}\big)-f\big(u_{\langle
n\rangle}^{<k}\big)\big\|_{ST^*(I)}\\\label{equ3.27}
&+\big\|f(u_{\langle n\rangle}^{<k})-\sum\limits_{j<k}f(u_{\langle
n\rangle}^{j})\big\|_{ST^*(I)}.
\end{align}

First, we estimate \eqref{equ3.26}. Let
$[G](I)=L_t^{3(d+1)}(I;L_x^\frac{6d(d+1)}{3d^2-3d-8}).$ It follows
from \eqref{fullst2} and the H\"older inequality that
($q=\frac{2(d+1)}{d-1}$)
\begin{align*}
\eqref{equ3.26}\leq&\big\|f\big(u_{(n)}^{<k}\big)-f\big(u_{\langle
n\rangle}^{<k}\big)\big\|_{L^{q'}(I;\dot{B}^\frac12_{q',2})}+\big\|f\big(u_{(n)}^{<k}\big)-f\big(u_{\langle
n\rangle}^{<k}\big)\big\|_{L_{t,x}^{q'}(I\times\R^d)}\\
\lesssim&\big\|u_{(n)}^{<k}-u_{\langle
n\rangle}^{<k}\big\|_{ST_\infty^j(I)}^\frac{2}{d-1}\big(\big\|(u_{(n)}^{<k},u_{\langle
n\rangle}^{<k})\big\|_{L_t^\infty\dot{H}_x^1\cap
ST_\infty^j(I)}\big)^{3-\frac2{d-1}}\\&+|I|^\frac12\big\|u_{(n)}^{<k}-u_{\langle
n\rangle}^{<k}\big\|_{[G](I)}\big(\big\|(u_{(n)}^{<k},u_{\langle
n\rangle}^{<k})\big\|_{[G](I)}\big)^2\\
\lesssim&\Big(\sum_{j<k,h_\infty^j=0}
\big\|\langle\nabla/h_n^j\rangle^{-2}\hat{U}_\infty^{j}\big\|_{ST_\infty^j(I)}\Big)^\frac{2}{d-1}\big(\big\|(u_{(n)}^{<k},u_{\langle
n\rangle}^{<k})\big\|_{L_t^\infty\dot{H}^1\cap
ST_\infty^j(I)}\big)^{3-\frac2{d-1}}\\&+|I|^\frac12\sum_{j<k,h_\infty^j=0}\big\|\langle\nabla/h_n^j\rangle^{-2}\hat{U}_\infty^{j}\big\|_{[G](I)}
\big(\big\|(u_{(n)}^{<k},u_{\langle
n\rangle}^{<k})\big\|_{[G](I)}\big)^2\\
\rightarrow&0,~~\text{as}~~n\to\infty,
\end{align*}
where we utilize \eqref{equ1.6.2} in the last second inequality and
the fact $\hat{U}_\infty^{j}\in L_t^\infty\dot{H}_x^1\cap
ST_\infty^j(I)\subset[G](I)$.

Next we estimate  $\eqref{equ3.27}.$  For $R\gg1$, we define
$\hat{U}_{n,R}^j$ by
\begin{equation}
\hat{U}_{n,R}^j(t,x)=\chi_R(t,x)\hat{U}_\infty^j(t,x)\prod_{l<k}\big\{(1-\chi_{h_n^{j,l}R})(t-t_n^{j,l},x-x_n^{j,l})~|~h_n^{j,l}<R^{-1}\big\},
\end{equation}
where
$(h_n^{j,l},t_n^{j,l},x_n^{j,l})=(h_n^l,t_n^j-t_n^l,x_n^j-x_n^l)/h_n^j,$
and $\chi_R(t,x)=\chi(\frac{t}{R},\frac{x}{R})$ with $\chi(t,x)\in
C^\infty_c(\mathbb{R}^{d+1})$ satisfing $\chi(t,x)=1$ for
$|(t,x)|\leq 1$ and $\chi(t,x)=0$ for $|(t,x)|\geq 2$. Then, noting
that either $h_n^{j,l}\to0$ or $|t_n^{j,l}|+|x_n^{j,l}|\to\infty$ by
\eqref{equ3.4}, we obtain $\hat{U}_{n,R}^j\to\chi_R\hat{U}_\infty^j$
in $ST_\infty^j(\R)$ and $[G](\R)$ as $n\to\infty$. Furthermore, we
get $\chi_R\hat{U}_\infty^j\to\hat{U}_\infty^j$ in the same spaces.

Therefore, we may replace $u_{\langle n\rangle}^j$ by $u_{\langle
n\rangle,R}^j:=h_n^jT_n^j\hat{U}_{n,R}^j\big((t-t_n^j)/h_n^j\big)$.
By the support property of $u_{\langle n\rangle,R}^j$, we have for
large $n$
\begin{equation}
\big(\sum_{j<k}u_{\langle
n\rangle,R}^{j}\big)^2=\sum_{j<k}\big|u_{\langle
n\rangle,R}^{j}\big|^2.
\end{equation}
Thus, we obtain
\begin{align*}
&\big\|f(u_{\langle n\rangle,R}^{<k})-\sum\limits_{j<k}f(u_{\langle
n\rangle,R}^{j})\big\|_{ST^*(I)}\\
\leq&\sum_{j\neq l}\big\|\big(V(\cdot)\ast\big|u_{\langle
n\rangle,R}^{j}\big|^2\big)u_{\langle n\rangle,R}^{l}\big\|_{ST^*(I)}\\
=&\sum_{j\neq
l}\big(h_n^{j,l}\big)^{1-\frac{d}2}\big\|\big(V(\cdot)\ast\big|\hat{U}_{\langle
n\rangle,R}^{j}\big|^2\big)(t,x)\hat{U}_{\langle n\rangle,R}^{l}\big(\frac{t-t_n^{j,l}}{h_n^{j,l}},\frac{x-x_n^{j,l}}{h_n^{j,l}}\big)\big\|_{ST^*(I)}\\
\to&\quad 0,~\text{as}~n\to\infty,
\end{align*}
by Lebesgue domainted convergence theorem, since either
$h_n^{j,l}\to0$ or $|t_n^{j,l}|+|x_n^{j,l}|\to\infty$ by
\eqref{equ3.4}. Thus we concludes the proof of Lemma \ref{lem3.3}.

\end{proof}
After this preliminaries, we now show that
$\vec{u}_{(n)}^{<k}+\vec{\omega}_n^k$ is a good approximation for
$\vec{u}_n$ provided that each nonlinear profile has finite global
Strichartz norm.
\begin{lemma}\label{precldes}
 Let $u_n$ be a sequence of local solutions of
\eqref{equ1} around $t=0$ satisfying
$\varlimsup\limits_{n\rightarrow\infty}E(u_n,\dot{u}_n)<+\infty.$
Assume that in its nonlinear profile decomposition
\eqref{nonlineard},  we have for any $j$
\begin{equation}\label{equ3.11}
\|\hat{U}^j_\infty\|_{ST_\infty^j(\mathbb{R})}+\|\vec{U}^j_\infty\|_{L_t^\infty
L_x^2(\mathbb{R})}<\infty.
\end{equation}
Then, for large $n$, $u_n$ is bounded in the Strichartz and the
energy norms, that is
\begin{equation}\label{equ3.12}
\varlimsup\limits_{n\rightarrow\infty}\big(\|u_n\|_{ST(\mathbb{R})}+\|\vec{u}_n\|_{L_t^\infty
L_x^2(\mathbb{R}\times \R^d)}\big)<+\infty.
\end{equation}

\end{lemma}

\begin{proof}
We only need to verify the conditions of Lemma \ref{long}. For this
purpose, we always use the fact that $u_{(n)}^{<k}+\omega_n^k$
satisfies that
\begin{align*}
(i\partial_{t}+\langle\nabla\rangle)\big(\vec{u}_{(n)}^{<k}+\vec{\omega}_n^k\big)
=-&f(u_{(n)}^{<k}+\omega_n^{k})+eq\big(u_{(n)}^{<k},\omega_n^{k}\big),
\end{align*}
where the error term $eq\big(u_{(n)}^{<k},\omega_n^{k}\big)$ is
\begin{align*}
eq\big(u_{(n)}^{<k},\omega_n^{k}\big)=
&\sum_{j<k}\big(\langle\nabla\rangle-\langle\nabla\rangle_\infty^j\big)\vec{u}_{(n)}^j+f(u_{(n)}^{<k})-\sum_{j<k}f\big(
u_{\langle
n\rangle}^j\big)\\&+f(u_{(n)}^{<k}+\omega_n^{k})-f\big(u_{(n)}^{<k}\big),
\end{align*}
and $u_{\langle
n\rangle}^j=\big(\langle\nabla\rangle_\infty^j\big)^{-1}\langle\nabla\rangle
u_{(n)}^j$ is as before.

 First, by the definition of the nonlinear
concentrating wave $u_{(n)}^j$ and \eqref{lea}, we have
\begin{align}\nonumber
\big\|\big(\vec{u}_{(n)}^{<k}(0)+\vec{w}_n^k(0)\big)-\vec{u}_n(0)\big\|_{L^2_x}&\leq\sum\limits_{j=0}^{k-1}
\big\|\vec{u}_{(n)}^j(0)-\vec{u}_n^j(0)\big\|_{L^2_x}\rightarrow0,
\end{align}
as $n\rightarrow +\infty.$

Next, by the linear profile decomposition in Lemma \ref{lem3.1}, we
get
\begin{equation}\label{ee5}
\|\vec{u}_n(0)\|_{L^2}^2=\|\vec{v}_n(0)\|_{L^2}^2\geq\sum\limits_{j=0}^{k-1}\|\vec{v}_n^j(0)\|_{L^2}^2+o_n(1)
=\sum\limits_{j=0}^{k-1}\|\vec{u}_{(n)}^j(0)\|_{L^2}^2+o_n(1).
\end{equation}
Thus, using the small data  scattering(Lemma \ref{small}), we obtain
that except for a finite set $J\subset\mathbb{N}$, the energy of
$u_{(n)}^j$ with $j\not\in J$ is smaller than the iteration
threshold. Hence
$$\|u_{(n)}^j\|_{ST(\mathbb{R})}\lesssim\|\vec{u}_{(n)}^j(0)\|_{L^2_x},\quad
j\not\in J.$$ This together with \eqref{equ3.8}, \eqref{equ3.9},
\eqref{equ3.11} and \eqref{ee5} yields that for any finite interval
$I$
\begin{align}\nonumber
\sup\limits_{k}\varlimsup\limits_{n\rightarrow\infty}\|u_{(n)}^{<k}\|_{ST(I)}^2&\lesssim\sum\limits_{j\in
J}\|u_{(n)}^{j}\|_{ST(\mathbb{R})}^2+\sum\limits_{j\not\in
J}\|u_{(n)}^j\|_{ST(\mathbb{R})}^2\\\label{ee6}
&\lesssim\sum\limits_{j\in
J}\|\hat{U}_\infty^j\|_{ST_\infty^j(\mathbb{R})}^2+\varlimsup\limits_{n\rightarrow\infty}\|\vec{u}_n(0)\|_{L^2}^2<+\infty.
\end{align}
Combining this with the Strichartz estimate for $\omega_n^k$, we get
$$\sup\limits_{k}\varlimsup\limits_{n\rightarrow\infty}\|u_{(n)}^{<k}+\omega_n^k\|_{ST(I)}<+\infty.$$
By Lemma \ref{lem3.1} and Lemma \ref{lem3.3}, we have
$$\|f(u_{(n)}^{<k}+\omega_n^k)-f(u_{(n)}^{<k})\|_{ST^*(I)}\rightarrow
0,$$ and
$$\big\|f(u_{(n)}^{<k})-\sum
\limits_{j=0}^{k-1}f(u_{(n)}^j)\big\|_{ST^*(I)}\rightarrow 0,$$ as
$n\rightarrow+\infty.$ On the other hand, the linear part in
$eq\big(u_{(n)}^{<k},\omega_n^{k}\big)$ vanishes when
$h_\infty^j=1$, and is controlled when $h_\infty^j=0$ by
\begin{align*}
\big\|\big(\langle\nabla\rangle-|\nabla|\big)\vec{u}_{(n)}^j\big\|_{L_t^1(I;L^2_x)}\lesssim&|I|\big\|\langle\nabla\rangle^{-1}\vec{u}_{(n)}^j\big\|_{L_t^
\infty(\R;L^2_x)}\\
\simeq&|I|\big\|\langle\nabla/h_n^j\rangle^{-1}\vec{U}_{\infty}^j\big\|_{L_t^
\infty(\R;L^2_x)}\\
\lesssim&|I|\Big( \big\|P_{\leq
(h_n^j)^\frac12}\vec{U}_{\infty}^j\big\|_{L_t^
\infty(\R;L^2_x)}+(h_n^j)^\frac12\big\|P_{>
(h_n^j)^\frac12}\vec{U}_{\infty}^j\big\|_{L_t^
\infty(\R;L^2_x)}\Big) \\
\to&0,\quad\text{as}\quad n\to\infty.
\end{align*}
 Thus,
$\big\|eq\big(u_{(n)}^{<k},\omega_n^{k}\big)\big\|_{ST^\ast(I)}\to0,$
as $n\to\infty.$

Therefore, for $k$ sufficiently close to $K$ and $n$ large enough,
the true solution $u_n$ and the approximate solution
$u_{(n)}^{<k}+\omega_n^k$ satisfy all the assumptions of the
perturbation   Lemma \ref{long}. Hence we can obtain the desired
result.
\end{proof}

\section{Concentration Compactness}
\setcounter{section}{4}\setcounter{equation}{0} By the profile
decomposition in the previous section and the perturbation theory,
we argue in this section that if the scattering result does not
hold, then there must exist a minimal energy solution with some good
compactness properties. This is the object of the following
proposition.
\begin{proposition}\label{prop4.1}
 Suppose that  $E_{max}<+\infty$. Then there exists a
global solution $u_c$ of \eqref{equ1} satisfying
\begin{equation}\label{critical1}
E(u_c)=E_{max},\quad \|u_c\|_{ST(\mathbb{R})}=+\infty.
\end{equation}
Moreover, there exists $c(t):\mathbb{R}^+\rightarrow\mathbb{R}^d$,
such that $K=\{(u_c,\dot{u}_c)(t,x-c(t))\ \big|\ t\in\mathbb{R}^+\}$
is precompact in $H^1(\R^d)\times L^2(\R^d)$. Besides, one can
assume that $c(t)$ is $C^1$ and satisfies
\begin{equation}\label{smooth}
|\dot{c}(t)|\lesssim_{u_c} 1
\end{equation}
uniformly in $t$.

\end{proposition}
\begin{proof}  The
proof of \cite{IMN} can be adapted verbatim, but we give a sketch
for the sake of completeness. By the definition of $E_{max}$, we can
choose a sequence $\{u_n(t)\}$ such that
\begin{equation}\label{defocusing} E(u_n,\dot{u}_n)\rightarrow
E_{max},\ \text{and}\ \|u_n\|_{ST(I_n)}\rightarrow\infty,\
\text{as}\ n\rightarrow\infty.\end{equation}  Now we consider the
linear and nonlinear profile decompositions of $u_n$, using Lemma
\ref{lem3.1},
\begin{align}\nonumber
e^{it\langle\nabla\rangle}\vec{u}_n(0)=\sum\limits_{j=0}^{k-1}
\vec{v}^j_n+\vec{\omega}_n^k,\
\vec{v}^j_n=e^{i\langle\nabla\rangle(t-t^j_n)}T_n^j\varphi^j(x),\\
u_{(n)}^{<k}=\sum\limits_{j=0}^{k-1}u_{(n)}^j,\
\vec{u}^j_{(n)}(t,x)=T_n^j\vec{U}_\infty^j\big((t-t_n^j)/h_n^j\big),\\\nonumber
\|\vec{v}_n^j(0)-\vec{u}_{(n)}^j(0)\|_{L^2_x}\rightarrow 0, \ as\
n\rightarrow\infty.
\end{align}
Lemma \ref{precldes} precludes that all the nonlinear profiles
$\vec{U}_\infty^j$ have finite global Strichartz norm. On the other
hand, every solution of \eqref{equ1} with energy less than $E_{max}$
has global finite Strichartz norm by the definition of $E_{max}$.
Hence by \eqref{equ3.5}, we deduce that there is only one profile,
i.e. $K=1,$ and so for large $n$
\begin{equation}
\tilde{E}(\vec{u}_{(n)}^0)= E_{max},\
\|\hat{U}_\infty^0\|_{ST_\infty^0(\R)}=\infty,\
\lim\limits_{n\rightarrow\infty}\|\vec{\omega}_n^1\|_{L_t^\infty
L^2_x}=0.
\end{equation}
If $h_n^0\to0,$ then $\hat{U}_\infty^0=\Re|\nabla|^{-1}\vec
U_\infty^0$ solves the $\dot H^1$-critical wave-Hartree equation
$$\partial_{tt}u-\Delta u+(|x|^{-4}\ast|u|^2)u=0$$
and satisfies
$$E(\hat U_\infty^0(\tau_\infty^0))=E_{max}<+\infty,~\big\|\hat
U_\infty^0\big\|_{L_t^q(\R;\dot{B}^\frac12_{q,2})}=\infty,~q=\frac{2(d+1)}{d-1}.$$
But Miao-Zhang-Zheng \cite{MZZ} has proven that there is no such
solution. Hence $h_n^0=1.$ And so there exist a sequence
$(t_n,x_n)\in \mathbb{R}\times\mathbb{R}^d$ and $\phi\in
L^2(\mathbb{R}^d)$ such that along some subsequence,
\begin{equation}\label{sequa}
\|\vec{u}_n(0,x)-e^{-it_n\langle\nabla\rangle}\phi(x-x_n)\|_{L^2_x}\rightarrow
0, \ n\rightarrow\infty.\end{equation}

Now we show that $\hat
U_\infty^0=\langle\nabla\rangle^{-1}\vec{U}_\infty^j$ is a global
solution. Assume not, then we can choose a sequence $t_n\in \R$
which approaches the maximal existence time. Since $\hat
U_\infty^0(t+t_n)$ satisfies \eqref{defocusing}, then applying the
above argument to it, we obtain by \eqref{sequa} that for some
$\psi\in L^2$ and another sequence $(t_n',x_n')\in
\mathbb{R}\times\mathbb{R}^d$ such that
\begin{equation}\label{sequa1}
\|\vec{U}_\infty^0(t_n)-e^{-it_n'\langle\nabla\rangle}\psi(x-x_n')\|_{L^2_x}\rightarrow
0, \end{equation} as $n\rightarrow\infty.$ Let
$\vec{v}:=e^{it\langle\nabla\rangle}\psi.$ For any $\varepsilon>0,$
there exist $\delta>0$ with $I=[-\delta,\delta]$ such that
$$\big\|\langle\nabla\rangle^{-1}\vec{v}(t-t_n')\big\|_{ST(I)}\leq\varepsilon,$$
which together with \eqref{sequa1} shows that for sufficiently large
$n$
$$\big\|\langle\nabla\rangle^{-1}e^{it\langle\nabla\rangle}\vec
U_\infty^0(t_n)\big\|_{ST(I)}\leq\varepsilon.$$ If $\varepsilon$ is
small enough, this implies that the solution $\vec U_\infty^0$
exists on $[t_n-\delta,t_n+\delta]$ for large $n$ by the small data
theory (Lemma \ref{small}). This contradicts the choice of $t_n$.
Thus $\hat U_\infty^0$ is a global solution and it is just the
desired critical element $u_c$. Moreover, since \eqref{equ1} is
symmetric in $t$, we may assume that
\begin{equation}\label{forward}
\|u_c\|_{ST(0,+\infty)}=+\infty.
\end{equation} We call such $u$ a forward critical element.

One can refer to \cite{MZh} for the  choice of  $c(t)$. Thus we
concludes the proof of Proposition \ref{prop4.1}.
\end{proof}

As a consequence of the above proposition and the
Hardy-Littlewood-Sobolev inequality, we have
\begin{corollary}\label{quick}
Let $u$ be a forward critical element. And we
denote$$E_{R,c}=\int_{|x-c|\geq R}\big(|u|^2+|\nabla
u|^2+|\dot{u}|^2\big)dx+\iint\limits_{\substack{|x-c|\geq R \\ y\in
\mathbb{R}^d}}\frac{|u(t,x)|^2|u(t,y)|^2}{|x-y|^4}dx dy,$$ then for
any $\eta>0$, there exists $R(\eta)>0$ such that
$$E_{R(\eta),c(t)}\leq\eta E(u,\dot{u}),\ for\ any\ t>0.$$
\end{corollary}

The next corollary is the conclusion of this section.
\begin{corollary}\label{cor4.2}
Let $u$ be a nonlinear strong solution of \eqref{equ1} such that the
set $K$ defined in Proposition \ref{prop4.1} is precompact in
$H^1\times L^2$, and $E(u,\dot{u})\neq 0.$ Then there exists a
constant $\beta=\beta(\tau)>0$ such that, for all time $t>0,$ there
holds that
\begin{align}\label{big}
\int_t^{t+\tau}\iint_{\mathbb{R}^d\times\mathbb{R}^d}\frac{|x_2-y_2|^2}{|x-y|^{6}}|u(s,x)|^2|u(s,y)|^2dx
dyds \geq\beta,
\end{align}
where  $x_2$ denotes the second component of $x\in\mathbb{R}^d$. In
particular, there holds that
\begin{equation}
\int_0^{t}\iint_{\mathbb{R}^d\times\mathbb{R}^d}\frac{|x_2-y_2|^2}{|x-y|^{6}}|u(t,x)|^2|u(t,y)|^2dx
dyds\gtrsim t.
\end{equation}
\end{corollary}

\begin{proof}
One can refer to \cite{MZh} for the detail proof.
\end{proof}

\section{Extinction of the critical element}
\setcounter{section}{5}\setcounter{equation}{0}   In this section,
we utilize the technique in \cite{Pa2} to prove that the critical
solution constructed in Section 4 does not exist, thus ensuring that
$E_{max} = +\infty$. This implies Theorem \ref{theorem}.

\begin{proposition}\label{infty} Assume that $d\geq 5,$  then $E_{max}=+\infty.$
\end{proposition}
\begin{proof}We use a Virial-type estimate in a direction orthogonal to the
momentum vector. Up to relabeling the coordinates, we might assume
that $\hbox{Mom}(u)$ is parallel to the first coordinate. Thus we
have
\begin{equation}\label{equ4.1}\int_{\mathbb{R}^d}u_t(t,x)\partial_j u(t,x)dx=0,\quad \forall
j\geq 2.\end{equation}
 Let $\phi_R(x)=\phi(x/R)$ where
$\phi(x)$ is a nonnegative smooth radial function such that supp
$\phi\subseteq B(0,2)$ and $\phi\equiv1$ in $B(0,1)$. We define the
Virial action
$$I(t)=\int_{\mathbb{R}^d}z_2\phi_R(z)\partial_2u(t,x)u_t(t,x)dx,$$
where $z=x-c(t)$ and $z_2$ denotes the second component of
$z\in\mathbb{R}^d$. Integrating by parts we get by \eqref{equ1}
\begin{align*}\nonumber
\partial_tI(t)=&\int_{\mathbb{R}^d}\partial_t(z_2\phi_R(z))\partial_2u(t,x)u_t(t,x)dx+\frac12\int_{\mathbb{R}^d}z_2\phi_R(z)
\partial_2(u_t(x,t))^2dx\\\nonumber
&+\int_{\mathbb{R}^d}z_2\phi_R(z)\partial_2u(t,x)\big(\Delta
u-u-(V(\cdot)*|u|^2)u\big)dx\\
=&\frac12\int_{\mathbb{R}^d}\big(-|u_t|^2+|u|^2+|\nabla
u|^2+(V(\cdot)*|u|^2)|u|^2\big)dx-\int_{\mathbb{R}^d}|\partial_2u|^2dx\\
&+\dot{z}_2\int_{\mathbb{R}^d}u_t\partial_2udx-2\int_{\mathbb{R}^d}z_2\phi_R(z)|u|^2(\frac{x_2}{|x|^{6}}*|u|^2))dx\\
 &+\int_{|z|\geq R}\mathcal{O}_1(u)dx,
\end{align*}
where
\begin{align*}
\mathcal{O}_1(u)=&\frac12\big[\frac{z_2}{R}\phi_{R}'-(1-\phi_R(x))\big]\big[-|u_t|^2+|u|^2+|\nabla
u|^2+(V(\cdot)*|u|^2)|u|^2\big]\\
&-(c'(t)\cdot\nabla\phi_R)\frac{z_2}{R}\partial_2 u
u_t-c'_2(t)(1-\phi_R(z))\partial_2u u_t-(\nabla\phi_R\cdot\nabla
u)z_2\partial_2u,
\end{align*}
is supported on the set $|z|\geq R$ and satisfies
 $$\big|\int_{|z|\geq R}\mathcal{O}_1(u)dx\big|\lesssim \int_{|z|\geq R}\big(|u|^2+|\nabla
u|^2+|\dot{u}|^2\big)dx.$$ Besides, we define the equirepartition of
energy action
$$J(t)=\int_{\mathbb{R}^d}\phi_R(z)u(t,x)u_t(t,x)dx.$$
Then
\begin{equation}
\partial_t J(t)=\int_{\mathbb{R}^d}\Big(|u_t|^2-|u|^2-|\nabla
u|^2-(V(\cdot)*|u|^2)|u|^2\Big)dx+\int_{|z|\geq
R}\mathcal{O}_2(u)dx,
\end{equation}
where
$$\mathcal{O}_2(u)=(1-\phi_R(z))\big[|u_t|^2-|u|^2-|\nabla
u|^2-(V(\cdot)*|u|^2)|u|^2\big]+\big(c^\prime(t)\cdot\nabla\phi_R\big)\frac{uu_t}{R}-\frac{u}{R}\nabla\phi_R\cdot\nabla
u,$$ has the same properties as $\mathcal{O}_1(u)$.

Considering $A(t)=I(t)+\frac12 J(t)$, we get
\begin{equation}\label{equ5.7}
|A(t)|\lesssim R E(u,\dot{u}),\ \text{for all time}\ t,
\end{equation}
 and
\begin{align*}
\partial_tA(t)=&-\int_{\mathbb{R}^d}|\partial_2u|^2dx-2\iint\limits_{\mathbb{R}^d\times\mathbb{R}^d}\phi_R(x-c(t))
(x_2-c_2(t))\frac{x_2-y_2}{|x-y|^{6}}|u(t,x)|^2|u(t,y)|^2dxdy
\\&-\int_{|z|\geq R}(\mathcal{O}_1(u)+\frac12\mathcal{O}_2(u))dx.
\end{align*}
And so by symmetrization, $\partial_tA(t)$ can be rewritten as
\begin{align}\nonumber
-\partial_tA(t)=&\int_{\mathbb{R}^d}|\partial_2u|^2dx+\iint_{\mathbb{R}^d\times\mathbb{R}^d}
\frac{|x_2-y_2|^2}{|x-y|^{6}}|u(t,x)|^2|u(t,y)|^2dxdy\\\label{const2}
&+I_2+\int_{|z|\geq R}(\mathcal{O}_1(u)+\mathcal{O}_2(u))dx,
\end{align}
where
\begin{align*}
I_2=\int\limits_{\mathbb{R}^d\times\mathbb{R}^d}&\big[(x_2-c_2(t))\phi_R(x-c(t))-(y_2-c_2(t))\phi_R(y-c(t))-(x_2-y_2)\big]\\
&\times\frac{x_2-y_2}{|x-y|^{6}}|u(t,x)|^2|u(t,y)|^2dxdy.
\end{align*}

We will show that $I_2$ constitute only a small fraction of
$E(u,u_t)$. First, by Corollary \ref{quick}, we know that if $R$ is
sufficient large depending on $u$ and $\eta$, then
$$E_{R,c(t)}(u,u_t)\leq \eta E(u,u_t).$$

Let $\chi$ denote a smooth cutoff to the region
$|x-c(t)|\geq\frac{R}{2}$ such that $\nabla\chi$ is bounded by
$R^{-1}$ and supported where $|x-c(t)|\sim R$. In the region where
$|x-c(t)|\thicksim|y-c(t)|$, we have
$$|x-c(t)|\thicksim|y-c(t)|\gtrsim R,$$ since otherwise $I_2$
vanish. Moreover, note that
$$|(x_2-c_2(t))\phi(x-c(t))-(y_2-c_2(t))\phi(y-c(t))|\lesssim|x-y|,$$
we use the Hardy-Littlewood-Sobolev inequality and Sobolev embedding
thoerem to control the contribution to $I_2$ from this regime by
\begin{align*}
\iint_{\mathbb{R}^d\times \mathbb{R}^d}\frac{|\chi u(t,x)|^2|\chi
u(t,y)|^2}{|x-y|^4}dxdy \lesssim\|\nabla(\chi u)\|_2^4\lesssim
\eta^2.
\end{align*}
In the region where $|x-c(t)|\ll|y-c(t)|,$ we use the fact that
$$|x-c(t)|\ll|y-c(t)|\sim|x-y|\ \text{and}\ |y-c(t)|\gtrsim R$$
to estimate the contribution from this regime by
$$\iint_{\mathbb{R}^d\times
\mathbb{R}^d}\frac{1}{|x-y|^4}|\chi u(t,y)|^2|u(t,x)|^2
dxdy\lesssim\|\nabla(\chi u)\|_{L_x^2}^2\|\nabla
u\|_{L_x^2}^2\lesssim \eta.$$ The last line follows from the same
computation as the first case. Finally, since the remaining region
$|y-c(t)|\ll|x-c(t)|$ can be estimated in the same way, we conclude
that
$$I_2\lesssim \eta.$$

Chosen $\eta$ sufficiently small depending on $u$ and $R$
sufficiently large depending on $u$ and $\eta$, we obtain
\begin{equation}\label{contrad}
-\partial_tA(t)\geq\iint_{\mathbb{R}^d\times\mathbb{R}^d}\frac{|x_2-y_2|^2}{|x-y|^{6}}|u(t,x)|^2|u(t,y)|^2dxdy-\eta
E(u,u_t).
\end{equation}
If $E_{max}<\infty$, integrating \eqref{contrad} from 0 to $T>0$ and
using Corollary \ref{cor4.2}, we get that there exists
$\alpha=\alpha(1,u)>0$ such that
$$\int_0^T\iint_{\mathbb{R}^d\times\mathbb{R}^d}\frac{|x_2-y_2|^2}{|x-y|^{6}}|u(s,x)|^2|u(s,y)|^2dx
dy ds\geq \alpha T,$$ for all $T>1.$ Thus $-A(t)\gtrsim T$ for large
$T$, which contradicts with \eqref{equ5.7}. Hence we have
$E_{max}=+\infty$, this concludes the proof of Proposition
\ref{infty}.
\end{proof}

\textbf{Acknowledgements:} The authors would like to express their
gratitude to the anonymous referee for his/her useful suggestions
and comments. Jiqiang Zheng was partly supported by ERC grant
SCAPDE.

\begin{center}

\end{center}

\end{document}